\documentclass[a4paper, 11pt]{amsart} 
\usepackage{geometry}
\usepackage{fullpage}

\usepackage{amssymb,amsmath,amsthm,ascmac,array,bm}
\usepackage{graphicx}
\usepackage{color}
\usepackage{enumerate}
\usepackage{mathrsfs} 
\usepackage{url}

\usepackage{tikz}
\usetikzlibrary{intersections,calc,arrows.meta}
\usetikzlibrary{quotes}
\usetikzlibrary{knots}

\usepackage{comment}
\usepackage{hyperref}

\usepackage{xcolor}
\usepackage[capitalize,nameinlink,noabbrev,nosort]{cleveref}
\newtheorem{theoremcounter}{Theorem Counter}[section]

\theoremstyle{definition}
\newtheorem{defn}[theoremcounter]{Definition}
\newtheorem{rem}[theoremcounter]{Remark}

\theoremstyle{plain}
\newtheorem{lem}[theoremcounter]{Lemma}
\newtheorem{prop}[theoremcounter]{Proposition}

\newtheorem{thm}[theoremcounter]{Theorem}

\numberwithin{equation}{section}

\allowdisplaybreaks 

\makeatletter
\@namedef{subjclassname@2020}{
\textup{2020} Mathematics Subject Classification}
\makeatother 



\begin{document}

%------------------------------------------------------------------

    %%\Large % 字体由小到大：

    %%abc\tiny\scriptsize\footnotesize\small\normalsize\large\Large\LARGE\huge\Huge

% 生成??%

%   \title{\LARGE \bf 
   \title{On Hasse norm principle for $3$-manifolds\\ in arithmetic topology}
   \author{Hirotaka TASHIRO}

\maketitle

%------------------------------------------------------------------

\begin{abstract}

Following the analogies between knots and primes, 3-manifolds and number rings in arithmetic topology, we show a topological analogue of the Hasse norm principle for finite cyclic coverings of 3-manifolds, which was originally stated for finite cyclic extensions of number fields.

\end{abstract}

%------------------------------------------------------------------

%\vspace{\baselineskip}

%\vskip\baselineskip

%\begin{center}
%\title\Large{\textbf{Introduction}}
%\end{center}

\section*{Introduction}

{\em Arithmetic topology} investigates the interaction between 3-dimensional topology and number theory, based on the analogies between knots and primes, 3-manifolds and number rings. We refer to \cite{Morishita2012} as a basic reference. In this paper, we present a new result in arithmetic topology, by showing a topological analogue of the Hasse norm principle for finite cyclic covers of 3-manifolds, which was initially established for finite cyclic extensions of number fields in \cite{Hasse1931Beweis}.

First, we marshal basic analogies in arithmetic topology, which will be used in this paper.

{\small 
\begin{center}

\begin{tabular}{|c|c|}

\hline

oriented connected closed 3-manifold & ring of integers of a number field $F$ \\

  $M$ & ${\rm Spec}(\mathcal{O}_{F})$\\

\hline

knot in $M$ & maximal ideal of $\mathcal{O}_F$ \\

$K:S^{1}\hookrightarrow S^{3}$ &${\rm Spec}(\mathcal{O}_F/{\mathfrak{p}})\hookrightarrow {\rm Spec}(\mathcal{O}_{F})$\\

\hline

link in $M$ & finite set of maximal ideals of $\mathcal{O}_F$ \\

$L=K_{1}\cup \cdots \cup K_{n}$ &$S=\{\mathfrak{p}_{1}, \cdots \mathfrak{p}_{n}\}$\\

\hline

tubular neighborhood of $K$  & $\mathfrak{p}$-adic integers \\

$V_{K}$ &${\rm Spec}(\mathcal{O}_{\mathfrak{p}})$\\

\hline

boundary torus of $V_{K}$  & $\mathfrak{p}$-adic field  \\

$\partial V_{K}$ &${\rm Spec}(F_{\mathfrak{p}})$\\

\hline

group of $2$-chains of $M$& multiplicative group of $F$ \\

$C_2(M; \mathbb{Z})$ & $F^{\times}$\\

\hline

group of $1$-cycles of $M$ & group of fractional ideals of ${\mathcal O}_{F}$ \\

$Z_1(M)$ & $J_F$ \\

\hline

boundary map & boundary map \\

$\partial _{M}:C_{2}(M) \rightarrow  Z_{1}(M)$ & $\partial _{F}:F^{\times} \rightarrow J_{F}$\\

$\Sigma \mapsto \partial _M \Sigma$ & $a \mapsto (a)$\\

\hline

1st homology group of $M$ & ideal class group of $F$\\

$H_{1}(M) = {\rm Coker}(\partial _{M})$ & $H_{F} = {\rm Coker}(\partial _{F})$\\

\hline

2nd homology group of $M$ & unit group of $\mathcal{O}_F$\\

$H_{2}(M)$ & $\mathcal{O}_F^{\times}$\\

\hline

finite (branched) covering & finite extension \\

$h:N\rightarrow M$ &$E/F$\\

\hline

\end{tabular}

\end{center}
}
We also have the following analogy between the Hurewicz isomorphism and the Artin reciprocity in unramified class field theory.

\begin{center}

\begin{tabular}{|c|c|}

\hline

$H_{1}(M)\cong {\rm Gal}(M^{\rm ab}/M)$& $H_{F}\cong {\rm Gal}(F_{\rm ur}^{\rm ab}/F)$\\

\hline

\end{tabular}
\end{center}
Here $M^{\rm ab}$ (resp.~$F_{\rm ur}^{\rm ab}$) denotes the maximal abelian covering of $M$ (resp.~the maximal abelian unramified extension of $F$).

Let us recall the Hasse norm principle in number theory. For a number field $F$, let $I_F$ and $P_F$ denote the id\`ele group and the principal id\`ele group, respectively. Note that $P_F$ is defined as the image of the diagonal map $\Delta :F^\times \hookrightarrow I_F$. For a finite extension $E/F$, let $N_{E/F}:I_E \to I_F$ denote the norm map. We have the Hasse norm principle as follows. 

\begin{thm}[Hasse \cite{Hasse1931Beweis}] \label{thm.Hasse} 
Let $E/F$ be a finite cyclic extension of number fields. Then, 
\[
P_{F}\cap {\rm N}_{E/F}(I_{E})= {\rm N}_{E/F}(P_{E}).
\]
\end{thm}

Now, let us turn to a topological analogue for 3-manifolds. Let $M$ be an oriented, connected, closed $3$-manifold. We fix a very admissible link $\mathcal{L}$ in $M$, which is a link consisting of countably many (finite or infinite) tame components with certain conditions and plays an analogous role to the set of primes, and we shall employ the notions of the id\`{e}le group $I_{M,\mathcal{L}}$ and the principal id\`{e}le group $P_{M,\mathcal{L}}$ introduced in \cite{NiiboUeki}. Note that $P_{M,\mathcal{L}}$ is defined as the image of the diagonal map $\Delta _{M,\mathcal{L}}:H_2 (M,\mathcal{L})\to I_{M,\mathcal{L}}$. For a finite cover $f: N\rightarrow M$ branched over a finite sublink $L$ of $\mathcal{L}$,  $f^{-1}(\mathcal{L})$ is again a very admissible link of $N$. Note that the induced map $f_* : I_{N,f^{-1}(\mathcal{L})} \rightarrow I_{M,\mathcal{L}}$ is an analogue of the norm map. Our main theorem, which may be regarded as a topological analogue of the Hasse norm principle,  is stated as follows.

\setcounter{section}{3}
\setcounter{theoremcounter}{0}
\begin{thm}% [Theorem 3.1]
Let $M$ be an integral homology $3$-sphere endowed with a very admissible link $\mathcal{L}$. Let $f:N\to M$ be a finite cyclic covering branched over a finite sublink $L_0$ of $\mathcal{L}$. Then, 
\[
P_{M,\mathcal{L}}\cap f_{*}(I_{N, f^{-1}(\mathcal{L})}))=f_{*}(P_{N, f^{-1}(\mathcal{L})}).
\]
\end{thm}
\setcounter{section}{0}

This paper is organized as follows. In Section 1, we review a topological analogue of local class field theory for a knot in a 3-manifold, and introduce the id\`ele group $I_{M,\mathcal{L}}$ for a 3-manifold $M$ endowed with a very admissible link $\mathcal{L}$.  Then we recall the construction of the diagonal map $\Delta _{M,\mathcal{L}} : H_2 (M,\mathcal{L}) \rightarrow I_{M,\mathcal{L}}$ and introduce the principal id\`{e}le group $P_{M,\mathcal{L}}$. In Section 2, we give an explicit description of the diagonal $\Delta _{M,\mathcal{L}}$ assuming that $M$ is an integral homology 3-sphere by using Seifert surfaces. In Section 3, we prove our main result, namely, a topological analogue of the Hasse norm principle for a finite cyclic covering $N \rightarrow M$ branched over a finite link in $\mathcal{L}$.\\

\noindent 
\textbf{Notation and convention}. 3-manifolds are assumed to be oriented, connected, and closed. A knot is assumed to be tame. For a manifold $X$ and its submanifold $A$, we denote by $H_{n}(X)$ and $H_{n}(X,A)$ the $n$-th homology group and the $n$-th relative homology group with coefficients in $\mathbb{Z}$. 
The branch set of a branched covering of a 3-manifold is assumed to be a finite link. 
When $f:N\rightarrow M$ is a Galois covering, we denote by ${\rm Gal}(N/M)$ the Galois group of $f$. For a knot $K$ in a 3-manifold $M$, $\mu _{K}$ and $\lambda _{K}$ denotes the meridian and the (chosen) longitude of $K$ regarded as elements in $H_1(\partial V_K)$ and several other groups. When $M$ is an integral homology 3-sphere,  $\lambda _{K}$ denotes the preferred longitude of $K$. When $K, K'$ are disjoint knots in an integral homology 3-sphere $M$, then ${\rm lk}(K,K')$ denotes their linking number.

\section{Topological local class field theory and
ide\`eles} % for 3-manifolds}  
%the id\`ele and principal id\`{e}le groups for a 3-manifold} 
\label{s1} 
In this section, we recollect the topological local class field theory, the notion of a very admissible link $\mathcal{L}$ in a 3-manifold $M$, and the id\`ele  group and the principal id\`{e}le group of $(M,\mathcal{L})$. Also, we prove a lemma on meridians.  We consult \cite{Morishita2012}, \cite{Niibo1}, and \cite{NiiboUeki} as basic references for this section. For an arithmetic counterpart, we refer to \cite{Neukirch}.

\subsection{Local class field theory}
Let $M$ be an oriented, connected, closed 3-manifold and $K$ a knot in $M$. 
%, and let $K$ be a knot in $M$. 
Let $V_{K}$ be a tubular neighborhood of $K$. Let ${\partial V_{K}}^{\rm ab}$  denote the maximal (Abelian) covering of $\partial V_K$ and let ${\partial V_{K}}^{\rm ur}$ denote the covering of $\partial V_K$ which is the restriction of the maximal covering of $V_{K}$ to the boundary. Let $v_{K}:H_{1}(\partial V_{K})\rightarrow H_{1}(V_{K})$ denote the homomorphism induced by $\partial V_{K}\hookrightarrow V_{K}$, and $\rho_{K}:H_{1}(\partial V_{K}) \overset{\cong}{\to}{\rm Gal}({\partial V_{K}}^{{\rm ab}}/\partial V_{K})$ be the isomorphism which comes from the Hurewicz isomorphism.
Then a topological analogue of local class field theory for $\partial V_K$ is formulated as follows.

\begin{thm}
The following diagram of exact sequences commutes. 
%We have the following commutative diagram of exact sequences.
$$
\begin{array}{ccccccccc}

0 & \rightarrow & H_{2}(V_{K},\partial V_{K}) & \rightarrow & H_{1}(\partial V_{K}) & \stackrel{v_{K}}{\rightarrow} & H_{1}(V_{K}) & \rightarrow & 0 \\

& & %\rotatebox{90}{$\cong$} 
\cong \downarrow & & \cong \downarrow \rho_{K} & & \cong \downarrow & & \\

0 &\rightarrow& {\rm Gal}({\partial V_{K}}^{\rm ab}/{\partial V_{K}}^{{\rm ur}}) & \rightarrow & {\rm Gal}({\partial V_{K} }^{{\rm ab}}/\partial V_{K}) & \rightarrow & {\rm Gal}({\partial V_{K} }^{\rm ur}/\partial V_{K}) &\rightarrow & 0.

\end{array}
$$
%% xyを使ったほうが良いかも？
Here $H_2(V_K, \partial V_K) = \mathbb{Z}[\mu_K]$, $H_1(V_K) = \mathbb{Z}[\lambda_K]$, and $H_1(\partial V_K) \cong \mathbb{Z}[\mu_K] \oplus \mathbb{Z}[\lambda_K]$.

\end{thm}

\subsection{Id\`ele group and principal id\`ele group}
Next, we recall the notion of a very admissible link $\mathcal{L}$ of 3-manifold $M$ and the id\`ele and principal id\`ele groups of $(M,\mathcal{L})$.
\begin{defn}

Let $M$ be an oriented connected closed 3-manifold and let $\mathcal{L}$ be a link of $M$ consisting of countably many (finite or infinite) tame components. We call  $\mathcal{L}$  a {\em very admissible link} if $\mathcal{L}$ satisfies the following condition: For any finite cover $f :N\rightarrow M$ branched over a finite sublink $L$ of $\mathcal{L}$, $H_{1}(N)$ is generated by the homology classes of components of $f^{-1}(\mathcal{L})$.

\end{defn}
By Definition 1.2, if $\mathcal{L}$ is a very admissible link of $M$ and $f: N \rightarrow M$ is a finite covering branched over a finite sublink $L$ of $\mathcal{L}$, then $f^{-1}(\mathcal{L})$ is again a very admissible link of $N$.  The following theorem is fundamental.

\begin{thm}[cf.~{\cite[Theorem 2.3]{NiiboUeki}}] 
Any oriented, connected, closed 3-manifold $M$ contains a very admissible link $\mathcal{L}$.  
\end{thm}
Hereafter, we fix a very admissible link $\mathcal{L}$ in a 3-manifold $M$. We may assume that $\mathcal{L}$ is endowed with tubular neighborhoods. Indeed, by the method of blow-up, we may essentially assume that $\mathcal{L}$ is endowed with a tubular neighborhood $V_\mathcal{L}=\sqcup_{K\subset \mathcal{L}} V_K$, which is the disjoint union of tori (cf.~\cite{NiiboUeki2023RMS}). Note that we may instead consider families of tubular neighborhoods with natural identifications of their groups (cf.~\cite{NiiboUeki}), or work over the system of formal tubular neighborhoods as well (cf.~\cite{Mihara2019Canada}).

\begin{defn} We define the {\em id\`ele group} of $(M,\mathcal{L})$ by
$$I_{M,\mathcal{L}}:=\{(a_{K})_{K}\in \underset{K\subset \mathcal{L}}{\prod}H_{1}(\partial V_{K})\mid v_{K}(a_{K})=0 \mbox{ for almost all} \  K \subset \mathcal{L}\},$$
where $K \subset \mathcal{L}$ runs through all components of $\mathcal{L}$. An element of $I_{M,\mathcal{L}}$ is called an {\em id\`{e}le} of $(M,\mathcal{L})$.
\end{defn}

In order to define the principal id\`ele group, we define the $H_{2}(M,\mathcal{L})$ and construct the $\Delta _{M,\mathcal{L}}:H_2(M,\mathcal{L})\to I_{M,\mathcal{L}}$. For any finite sublinks $L,L'$ of $ \mathcal{L}$ with $L\subset L'$, there exists a natural injection $j_{L,L'}:H_{2}(M,L)\hookrightarrow H_{2}(M,L')$. Then $((H_2(M,L))_{L \subset \mathcal{L}}, (j_{L,L'})_{L \subset L'\subset \mathcal{L}})$ forms a direct system.\\

\begin{defn}
We define $H_{2}(M,\mathcal{L})$ by the following direct limit:
$$
H_{2}(M,\mathcal{L}):=\underset{L\subset \mathcal{L}}{\varinjlim}H_{2}(M,L)
=\underset{L\subset \mathcal{L}}{\bigsqcup}H_{2}(M,L)/\thicksim.\\
$$
Here, for $S_{L}\in H_{2}(M,L)$ and $S_{L'}\in H_{2}(M,L')$, we write $S_{L}\thicksim S_{L'}$ if and only if there exists a finite link $ L''\subset \mathcal{L}$ such that $L\cup L'\subset L'' \ \mbox{and}\   j_{L,L''}(S_{L})=j_{L',L''}(S_{L'})$.
\end{defn}
Let $L$ be a finite link of $\mathcal{L}$. We put $V_{L}:=\bigsqcup _{K\subset L}V_{K}$ and $X_{L}:=M\setminus {\rm Int}(V_L)$ $(\simeq M \setminus L)$. By the excision isomorphism ${\rm exc}$ and the relative homology exact sequence for a pair $(M, V_{L})$, we obtain a sequence
$$
H_{2}(M,L)\overset{\cong}{\rightarrow} H_{2}(M,V_{L})\overset{{\rm exc}}{\rightarrow} H_{2}(X_{L},\partial V_{L})\overset{\partial}{\rightarrow} H_{1}(\partial V_{L}).
\leqno{(1.6)}
$$
Moreover, for any finite links $L,L'\subset \mathcal{L}$ with $L\subset L'$, we have a commutative diagram
$$
\begin{array}{ccc}

H_{2}(M,L') & \overset{\partial}{\rightarrow} & H_{1}(\partial V_{L'})\\

j_{L,L'}\uparrow &\circlearrowright & \downarrow p_{L',L} \\

H_{2}(M,L) & \overset{\partial}{\rightarrow} & H_{1}(\partial V_{L}).

\end{array}
\leqno{(1.7)}
$$
Taking the direct limit for $H_2(M,L)$ and the projective limit for $H_1(V_L)$ with respect to finite sublinks $L \subset \mathcal{L}$, we obtain a natural homomorphism $\Delta _{M,\mathcal{L}}:H_{2}(M,\mathcal{L})\rightarrow I_{M,\mathcal{L}}$ called the \emph{diagonal map}. 
\setcounter{theoremcounter}{7}
\begin{defn}
We define the {\em principal id\`ele group} of $(M,\mathcal{L})$ by
$$
P_{M,\mathcal{L}}:={\rm Im}(\Delta_{M,\mathcal{L}}).
$$
An element of $P_{M,\mathcal{L}}$ is called a principal id\`{e}le of $(M,\mathcal{L})$.
\end{defn}

Finally, for a finite sublink $L \subset \mathcal{L}$,  we define the subgroup $U^{L}$ of $I_{M,\mathcal{L}}$ by $U^L :=\prod _{K\subset \mathcal{L} \setminus L}\mathbb{Z}[m_{K}]$. The following proposition will be used in Section 3.

\begin{prop}[{\cite[Lemma 5.7]{NiiboUeki}}]
We have an isomorphism
$$
I_{M,\mathcal{L}}/(P_{M,\mathcal{L}}+U^{L})\cong H_{1}(X_{L}).
$$
\end{prop}
\subsection{A lemma on meridians}
\begin{lem}
Let $M$ be an oriented connected closed 3-manifold endowed with a very admissible link $\mathcal{L}$. Let $f:N\rightarrow M$ be a finite covering branched over a finite link $L_{0}\subset \mathcal{L}$. 
%We denote by 
Let $f_*:I_{N,f^{-1}(\mathcal{L})}\rightarrow I_{M,\mathcal{L}}$ denote 
the homomorphism induced by $f$.
Then, we have 
$$
f_{*}(\prod_{J\subset f^{-1}(\mathcal{L})}\mathbb{Z}[\mu_J])\subset \prod_{K\subset \mathcal{L}}\mathbb{Z}[\mu_{K}].
$$
\end{lem}
 
\begin{proof}
%First, 
Let $K$ be a knot in $\mathcal{L}$, let $J$ be a component of $f^{-1}(K)$, 
and let $f_{\#}:H_{1}(\partial V_{f^{-1}(K)})\rightarrow H_{1}(\partial V_{K})$ denote the homomorphism induced by $f|_{\partial V_{f^{-1}(K)}}:\partial V_{f^{-1}(K)}\rightarrow \partial V_{K}$. 
%First, for any $K\subset \mathcal{L}$, we define $L:=f^{-1}(K)$ and a component of $L$ by $J$. We also denote by $f_{\#}:H_{1}(\partial V_L)\rightarrow H_{1}(\partial V_{K})$ the homomorphism induced by $f|_{\partial V_L}:\partial V_L\rightarrow \partial V_{K}$. 
By the commutative diagram
$$
\begin{array}{ccc}
H_{1}(\partial V_{f^{-1}(K)}) & \hookrightarrow & I_N\\

f_{\#} \downarrow &\circlearrowright & \downarrow f_{*} \\

H_{1}(\partial V_K) & \hookrightarrow & I_M,
\end{array}
$$
we have that $f_*([\mu_J])\in H_1(\partial V_K)$, where $H_1(\partial V_K)$ is seen as a direct product component of $I_{M,\mathcal{L}}$. 

Let $V_J$ denote the component of $f^{-1}(V_K)$ containing $J$. 
Then, we have the following commutative diagram
$$
\begin{array}{ccc}

\partial V_J & \hookrightarrow & V_J\\

\downarrow &\circlearrowright & \downarrow \\

\partial V_K & \hookrightarrow & V_K.

\end{array}
\leqno{(*)}
$$
By the definition of Fox completion, the meridian $[\mu_J] \in H_1(\partial V_J)$ of $J$ generates Ker$(H_1(\partial V_J)\rightarrow H_1(V_J))$. 
So, we have the following commutative diagram of exact sequences 
$$
\begin{array}{ccccccccc}

0 &  \rightarrow &\langle [\mu_J]\rangle &\hookrightarrow& H_{1}(\partial V_J) & \rightarrow & H_{1}(V_J) & \rightarrow & 0 \\

& &  & & \downarrow f_{\#} & \circlearrowright & \downarrow f_{\#}& & \\

0 &  \rightarrow &\langle [\mu_K]\rangle &\hookrightarrow& H_{1}(\partial V_{K}) & \rightarrow & H_{1}(V_{K}) & \rightarrow & 0 \\

\end{array}
$$
induced by ($\ast$). 
Thus, 
we have $f_\#([\mu_J])\in \langle[\mu_K]\rangle ={\rm Ker}(H_1(\partial V_K)\rightarrow H_1(V_K))$. 
\end{proof}

\section{Explicit description of the diagonal map $\Delta_{M,\mathcal{L}}$} \label{s2}

In this section, we explicitly describe the diagonal map $\Delta_{M,\mathcal{L}}$. Hereafter, $M$ is supposed to be an integral homology 3-sphere endowed with a very admissible link $\mathcal{L}$ unless otherwise mentioned.

Let $L$ be a finite sublink of $\mathcal{L}$ with its tubular neighborhood $V_L = \bigsqcup_{K \subset L} V_K$,  where $K \subset L$ runs through the components of $L$. Let $S_K$ denote a Seifert surface of $K$. 
We will describe the composition map 
$$ 
\Delta_L :  H_{2}(M,L)  \cong H_{2}(M,V_{L}) \stackrel{{\rm exc}}{\rightarrow} H_2(X_L, \partial V_L)
 \overset{\partial}{\rightarrow}H_{1}(\partial V_{L}) = \prod_{K \subset L} H_1(\partial V_K)
$$ 
of the maps in (1.6) in an explicit manner.

\begin{lem} \label{lem.H2ML} 
The relative homology group $H_2(M,L)$ is the abelian group freely generated by the homology classes of the Seifert surfaces $S_K$ of components $K$ of $L$:
$$H_{2}(M,L) = \bigoplus_{K \subset L} \mathbb{Z}[S_K].$$
\end{lem}

\begin{proof} By the relative homology exact sequence 
$$
H_{2}(M)\rightarrow H_{2}(M,L) \overset{\partial}{\rightarrow} H_{1}(L)\rightarrow H_{1}(M)
$$
for $(M,L)$ and the assumption that $H_1(M)=H_2(M)=0$, 
we have that $\partial :H_2(M,L)\to H_1(L)$ is an isomorphism. 
Since $H_1(L) = \bigoplus_{K \subset L} H_1(K) = \bigoplus_{K \subset L} \mathbb{Z}[K]$ and $\partial$ is the boundary map, we obtain the assertion.
\end{proof}

Next, we describe the excision isomorphism. Let $K$ be a component of $L$. For $K' \subset (L \setminus K)$, let $D_{K'} := S_K \cap V_{K'}$, which is a disk on $S_K$ or an emptyset.  We define $A_K$ to be  $S_{K} \setminus \cup_{K' \subset (L \setminus K)} {\rm Int}(D_{K'})$, which is a punctured surface.

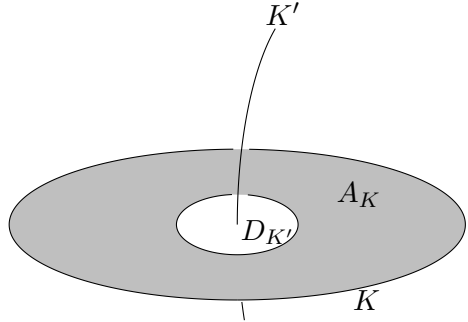
\begin{center}
  \begin{tikzpicture}
  %最初に文字の大きさで大小感覚をとる
    \node at (1.7,-1) { $K$};
\node at (0.6,2.8) { $K'$};

       %次に塗りつぶす
      \fill [lightgray] [domain=-90:270,variable=\t,samples=200] plot ({3*cos(\t)},{1*sin(\t)});
       \fill[white] [domain=95:440,variable=\t,samples=200] plot({0.8*cos(\t)},{0.4*sin(\t)});
       %更にA _{K}の部分になる楕円を描く
         \draw [domain=-90:87,variable=\t,samples=200] plot ({3*cos(\t)},{1*sin(\t)});
           \draw [domain=91:270,variable=\t,samples=200] plot ({3*cos(\t)},{1*sin(\t)});
       %左の線
         \draw [domain=120:180,variable=\t,samples=200] plot ({1+1*cos(\t)},{3*sin(\t)});
       
       \draw [domain=200:205,variable=\t,samples=200] plot ({1+1*cos(\t)},{3*sin(\t)});
      %A _{K}上の円
       \draw [domain=95:440,variable=\t,samples=200] plot ({0.8*cos(\t)},{0.4*sin(\t)});

      \node at (1.6,0.4) { $A_{K}$};
\node at (0.4,-0.1) { $D _{K'}$};
\end{tikzpicture}
\end{center}

The excision isomorphism ${\rm exc}:H_2(M,L)\overset{\cong}{\to} H_2(X_L,\partial X_L)$ sends $[S_K]$ to $[A_K]$. Hence we have the following. 
\begin{lem} \label{lem.H2XLVL} 
We have
$$ H_{2}(X_L, \partial V_L) = \bigoplus_{K \subset L} \mathbb{Z}[A_K].$$
\end{lem}

In order to describe the boundary map $\partial : H_2(X_L, \partial V_L) \rightarrow H_1(\partial V_L) = \bigoplus_{K \subset L} H_1(\partial V_K)$ explicitly, we give an orientation to each component $K \subset L$ so that $S_K$ has an orientation which is compatible with the orientation of $\partial S_K = K$. It gives an orientation of $\partial A_K$. For each component $K \subset L$, the orientation of the longitude $\lambda_K \subset \partial V_K$ is the same as $K$ and the orientation of the meridian $\mu_K \subset \partial V_K$ is determined by ${\rm lk}(\mu_K,K) = 1$. By the definition of $A_K$, we have the following. %(see the figure below). %cf. Fig.1). 

\begin{lem} \label{lem.derAK} 
We have %the following formula:
$$
 \partial ([A_{K}]) = [\lambda _{K}] - \sum_{K' \subset L \setminus K}{\rm lk}(K, K')[\mu_{K'}].
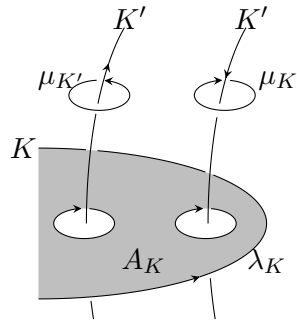
$$
%Here we note that if $K'$ does not intersect with $S_K$, then ${\rm lk}(K,K') = 0$.
\end{lem}

\begin{center}
\begin{tikzpicture}[>=stealth]

  %最初に文字の大きさで大小感覚をとる

    \node at (-0.2,1) { $K$};

\node at (3,-0.5) {$\lambda _{K}$};

\node at (1.2,2.73) { $K'$};

\node at (2.8,2.73) { $K'$};

\node at (0.3,1.9) { $\mu _{K'}$};

\node at (3.2,1.9) { $\mu _{K'}$};

       %次に塗りつぶす

      \fill [lightgray] [domain=-90:90,variable=\t,samples=200] plot ({3*cos(\t)},{1*sin(\t)});

       \fill[white] [domain=95:440,variable=\t,samples=200,->] plot({0.6+0.4*cos(\t)},{0.2*sin(\t)});

        \fill[white] [domain=95:440,variable=\t,samples=200,->] plot ({2.2+0.4*cos(\t)},{0.2*sin(\t)});

       %更にA _{K}の部分になる楕円を描く

         \draw [domain=-90:-45,variable=\t,samples=200,->] plot ({3*cos(\t)},{1*sin(\t)});

     \draw  [domain=-45:40,variable=\t,samples=200] plot ({3*cos(\t)},{1*sin(\t)});

      \draw [domain=43:75,variable=\t,samples=200] plot ({3*cos(\t)},{1*sin(\t)});

       \draw [domain=78:90,variable=\t,samples=200] plot ({3*cos(\t)},{1*sin(\t)});

       %左の線

         \draw [domain=120:135,variable=\t,samples=200] plot ({1.63+1*cos(\t)},{3*sin(\t)});

        \draw [domain=148:135,variable=\t,samples=200,->] plot ({1.63+1*cos(\t)},{3*sin(\t)});

        \draw [domain=180:151,variable=\t,samples=200] plot ({1.63+1*cos(\t)},{3*sin(\t)});

       \draw [domain=200:205,variable=\t,samples=200] plot ({1.63+1*cos(\t)},{3*sin(\t)});

       %右の線

            \draw [domain=120:140,variable=\t,samples=200,->] plot ({3.23+1*cos(\t)},{3*sin(\t)});

        \draw [domain=140:148,variable=\t,samples=200] plot ({3.23+1*cos(\t)},{3*sin(\t)});

  \draw [domain=152:180,variable=\t,samples=200] plot ({3.23+1*cos(\t)},{3*sin(\t)});

        \draw [domain=194:205,variable=\t,samples=200] plot ({3.23+1*cos(\t)},{3*sin(\t)});

%左のメリディアンの向きとA _{K}上の向き

            \draw [domain=95:440,variable=\t,samples=200,->] plot ({0.8+0.4*cos(\t)},{1.7+0.2*sin(\t)});%左上

       \draw [domain=440:95,variable=\t,samples=200,->] plot ({0.6+0.4*cos(\t)},{0.2*sin(\t)});%左下

   % 右のメリディアンの向きとA _{K}上の向き

\draw [domain=440:95,variable=\t,samples=200,->] plot ({2.45+0.4*cos(\t)},{1.7+0.2*sin(\t)});%右上

\draw [domain=440:95,variable=\t,samples=200,->] plot ({2.2+0.4*cos(\t)},{0.2*sin(\t)});%右下

 \node at (1.38,-0.5) { $A_{K}$};
% \node at (1.38,-1.5) {Fig.1};
\end{tikzpicture}
\end{center}

\begin{prop} \label{prop.DeltaSK}
Let $\sum_{K \subset L} c_K [S_K] \in H_2(M, L)$ with $c_K \in \mathbb{Z}$. Then we have
$$ \Delta_L (\sum_{K \subset L} c_K [S_K]) = \sum_{K \subset L} \Big(c_K [\lambda_K] - (\sum_{K' \subset L \setminus K} c_{K'} {\rm lk}(K,K')) [\mu_{K}]\Big).$$
\end{prop}

\begin{proof} 
Since ${\rm exc}:H_2(M,L)\overset{\cong}{\to} H_2(X_L,\partial X_L)$ maps $[S_K]\mapsto [A_K]$, by Lemmas \ref{lem.H2ML} and \ref{lem.H2XLVL}, it suffices to show that
$$ \partial  (\sum_{K \subset L} c_K [A_K]) = \sum_{K \subset L} (c_K [\lambda_K] - \sum_{K' \subset L \setminus K} c_{K'} {\rm lk}(K,K') [\mu_{K}]).$$ 
By \cref{lem.derAK}, the left-hand side satisfies that 
\begin{align*}
 \partial (\sum_{K \subset L} c_K [A_K])%\\
 =&\sum_{K\subset L}c_{K}\partial([A_{K}])\\
 =&\sum_{K\subset L}c_{K}([\lambda_{K}]+\sum_{K'\subset L\setminus K}-{\rm lk}(K,K')[\mu_{K'}])\\
 =&\sum_{K\subset L}c_{K}[\lambda_K]+\sum_{K\subset L}\sum_{K'\subset L\setminus K}-{\rm lk}(K,K')c_{K}[\mu_{K'}] \ \ \cdots (*).
\end{align*}
By the symmetry of the linking number, the second term of $(*)$ is computed as follows: 
\begin{align*}
\sum_{K\subset L}\sum_{K'\subset L\setminus K}-{\rm lk}(K,K')c_{K}[\mu_{K'}] %\\
=&\sum_{K,K'\subset L,\ K\neq K'} -{\rm lk}(K,K')c_{K}[\mu_{K'}]\\
=&\sum_{K,K'\subset L,\ K\neq K'} -{\rm lk}(K,K')c_{K'}[\mu_{K}]\\
=&\underset{K\subset L}{\sum}(\underset{K'\subset L\setminus K}{\sum}-{\rm lk}(K,K')c_{K'})[\mu_{K}].
\end{align*} 
Thus, we obtain the desired formula. 
\end{proof}
%
%
\begin{comment} 
Then the 2nd term of $(*)$ is computed as follows:
$$
\begin{array}{ll}
 &\underset{K_{i}\subset L}{\sum} \underset{K'\subset L\setminus K_{i}}{\sum}\Bigl(-{\rm lk}(K_{i},K')c_{K_{i}}[\mu _{K'}]\Bigr)\\
=&\underset{K_{i}\subset L}{\sum}\Bigl(-{\rm lk}(K_{i},K_{1})c_{K_{i}}[\mu _{K_{1}}]-\cdots -{\rm lk}(K_{i},K_{n})c_{K_{i}}[\mu _{K_{n}}]\Bigr)\ (\text{by\ }K_{i}\neq K')\\
=&-{\rm lk}(K_{1},K_{2})c_{K_{1}}[\mu_{K_{2}}]- \cdots -{\rm lk}(K_{1},K_{n})c_{K_{1}}[\mu_{K_{n}}]\\
 &+(-{\rm lk}(K_{2},K_{1})c_{K_{2}}[\mu_{K_{1}}]-\cdots -{\rm lk}(K_{2},K_{n})c_{K_{2}}[\mu_{K_{n}}])\\
 &+\cdots \\
 &+(-{\rm lk}(K_{n},K_{1})c_{K_{n}}[\mu_{K_{1}}]-\cdots -{\rm lk}(K_{n},K_{n-1})c_{K_{n}}[\mu_{K_{n-1}}])\\
=&\underset{K_{j}\subset L\setminus K_{1}}{\sum}-{\rm lk}(K_{j},K_{1})c_{K_{j}}[\mu_{K_{1}}]+\cdots +\underset{K_{j}\subset L\setminus K_{n}}{\sum}-{\rm lk}(K_{j},K_{n})c_{K_{j}}[\mu_{K_{n}}]\\
=&\underset{K_{i}\subset L}{\sum}\underset{K_{j}\subset L\setminus K_{i}}{\sum}-{\rm lk}(K_{j},K_{i})c_{K_{j}}[\mu _{K_{i}}]\\
=&\underset{K\subset L}{\sum}(\underset{K'\subset L\setminus K}{\sum}-{\rm lk}(K,K')c_{K'})[\mu_{K}].
\end{array}
\end{comment} 
%
%

\begin{rem}
By \cref{prop.DeltaSK}, we can explicitly verify the 
commutativity of the diagram (1.7) 
(i.e. $\partial =p_{L',L}\circ \partial \circ j_{L,L'}$) 
as follows: 
$$
\begin{array}{ll}
 &(p_{L',L}\circ \partial \circ j_{L,L'})(\underset{K\subset L}{\sum}c_{K}[S_{K}]) \\
=& (p_{L',L}\circ \partial)(\underset{K\subset L'}{\sum}c_{K}[S_{K}]) \;\;\;\; (\text{by\ }K\subset L'\setminus L\Rightarrow c_{K}=0)\\
=& p_{L',L}(\underset{K\subset L'}{\sum}(c_{K}[\lambda _{K}]-(\underset{K'\subset L\setminus K}{\sum}c_{K'}{\rm lk}(K,K'))[\mu _{K}])) \\
=&\underset{K\subset L}{\sum}(c_{K}[\lambda _{K}]-(\underset{K'\subset L\setminus K}{\sum}c_{K'}{\rm lk}(K,K'))[\mu _{K}]) \\
=&\partial ({\underset{K\subset L}{\sum}c_{K}[S_{K}]}).  \\
\end{array}
$$
\end{rem}
The following assertion explicitly describes the diagonal map $\Delta_{M,\mathcal{L}} : H_{2}(M,\mathcal{L})\rightarrow I_{M,\mathcal{L}}$.

\begin{prop}
Let $[A] \in H_{2}(M,\mathcal{L})$. Then there is a finite sublink $L\subset \mathcal{L}$ such that $[A] \in H_{2}(M,L)$. By \cref{lem.H2ML}, we can write
$[A] =\sum_{K \subset L}c_{K}[S_{K}]$ with $c_K \in \mathbb{Z}$.
Let  $\Delta_{M,\mathcal{L}}([A]) = (a_K)_{K \subset \mathcal{L}}\in I_{M,\mathcal{L}}$. Then we have the following formula:
$$
a_{K}=\begin{cases}
\displaystyle c_{K}[\lambda _{K}]- (\sum_{K'\subset L\setminus K} {\rm lk}(K,K')c_{K'})[\mu _{K}] & (K\subset L) \\
\displaystyle -\sum_{K'\subset L} {\rm lk}(K,K')c_{K'} [\mu _{K}] &(K\subset \mathcal{L}\setminus L)
\end{cases}
$$
\end{prop}

\begin{proof}
For the case that $K \subset L$, the formula follows from \cref{prop.DeltaSK}. 
For the case that $K \subset \mathcal{L} \setminus L$, the formula is obvious.
\end{proof}

\section{Hasse norm principle for 3-manifolds} \label{s3}

In this section, we prove the following topological analogue of the Hasse norm principle.

\begin{thm} \label{thm.HNP}
Let $M$ be an integral homology $3$-sphere endowed with a very admissible link $\mathcal{L}$. Let $f : N\rightarrow M$ be a finite cyclic covering branched over a finite sublink $L_{0}$ of $\mathcal{L}$. 
Then 
$$
P_{M,\mathcal{L}}\cap f_{*}(I_{N,f^{-1}(\mathcal{L})})=f_{*}(P_{N,f^{-1}(\mathcal{L})}).
$$
\end{thm} 

\begin{lem}[cf.~{\cite[Theorem 6.2.1]{Morishita2012}}]

Keeping the same notations as in Theorem 3.1, let $\tau$ be a generator the Galois group ${\rm Gal}(N/M)$, and let $X:=M\setminus L_{0}$ and $Y:=N\setminus f^{-1}(L_0).$ Then, we have the following exact sequence
$$
0\rightarrow (\tau -1)H_{1}(Y)\rightarrow H_{1}(Y)\overset{f_{*}}{\rightarrow}f_{*}(H_{1}(Y))\rightarrow 0.
$$
\end{lem}

\begin{proof}[Proof of Theorem 3.1.]
Set $\mathfrak{L}:=f^{-1}(\mathcal{L})$. By the following commutative diagram
$$
\begin{array}{ccc}
H_{2}(N,\mathfrak{L}) & \overset{\Delta _{N,\mathfrak{L}}}{\rightarrow} & I_{N,\mathfrak{L}}\\
f_{*}\downarrow &\circlearrowright & \downarrow f_{*} \\
H_{2}(M,\mathcal{L}) & \overset{\Delta_{M,\mathcal{L}}}{\rightarrow} & I_{M,\mathcal{L}},
\end{array}
$$
we have $P_{M,\mathcal{L}}\cap f_{*}(I_{N,\frak{L}})\supset f_{*}(P_{N, \frak{L}})$. Hence, it suffices to show $P_{M,\mathcal{L}}\cap f_{*}(I_{N,\frak{L}})\subset f_{*}(P_{N,\frak{L}})$.

Take any element $f_{*}(\mathfrak{a})\in f_{*}(I_{N,\mathfrak{L}})\cap P_{M,\mathcal{L}} (\mathfrak{a}\in I_{N,\mathfrak{L}})$. By using Proposition 1.9 and Lemma 3.2, we obtain the following exact sequence
%{\small
$$
0 \rightarrow \frac{(\tau -1)I_{N,\mathfrak{L}}}{P_{N,\mathfrak{L}}+U^{f^{-1}(L_{0})}} \rightarrow \frac{I_{N,\mathfrak{L}}}{P_{N,\mathfrak{L}}+U^{f^{-1}(L_{0})}}
\overset{f_{*}}{\rightarrow} \frac{f_{*}(I_{N,\mathfrak{L}})+P_{M,\mathcal{L}}+U^{L_{0}}}{P_{M,\mathcal{L}}+U^{L_{0}}} \rightarrow 0.
$$
%}
Since $f_{*}(\mathfrak{a})\in P_{M,\mathcal{L}}+U^{L_{0}},\mbox{ we have }\mathfrak{a}+P_{N,\mathfrak{L}}+U^{f^{-1}(L_{0})}\in {\rm Ker}(f_{*})$. 
By the above exact sequence, there exists $\mathfrak{c}\in I_{N,\mathfrak{L}}$ such that $\mathfrak{a} \equiv (\tau -1)\mathfrak{c} \; \mbox{mod} \; P_{N,\mathfrak{L}}+U^{f^{-1}(L_{0})}$. 
Thus, there exists $\alpha \in P_{N,\mathfrak{L}}\text{ and }\mathfrak{u}\in U^{f^{-1}(L_{0})}$ such that $\mathfrak{a}-(\tau -1)\mathfrak{c}=\alpha +\mathfrak{u}$. 
By $f\circ \tau =f$, we have $f_{*}(\mathfrak{a})=f_{*}(\alpha)+f_{*}(\mathfrak{u})$. 
Since $\alpha \in P_{N,\mathfrak{L}}$ and $f_{*}(\mathfrak{a})\in P_{M,\mathcal{L}}$, by Definition 1.8, there exist $[A]\in H_{2}(N,\mathfrak{L})$ and $[B]\in H_{2}(M,\mathcal{L})$ such that $\alpha =\Delta _{N,\mathfrak{L}}([A])\mbox{ and }f_{*}(\mathfrak{a})=\Delta _{M,\mathcal{L}}([B])$. Therefore, we have
$$
\begin{array}{cl}

f_{*}(\mathfrak{u})&=\ \Delta _{M,\mathcal{L}}([B])-(f_{*}\circ \Delta _{N,\mathfrak{L}})([A])\\

 &=\ \Delta _{M,\mathcal{L}}([B])-(\Delta _{M,\mathcal{L}}\circ f_{*})([A])\\

 &=\ \Delta _{M,\mathcal{L}}([B]-f_{*}([A])).

\end{array}
$$

Suppose $[B]-f_{*}([A])\ne 0$. Then there exists $L_{1}\subset \mathcal{L}$ such that $[B]-f_{*}([A])= \underset{K\subset L_{1}}{\sum}c_{K}[S_{K}]$ with $c_{K}\ne 0$ for some $K \subset L_1$. By using Proposition 2.6, we obtain 
%\small{
$$
f_*(\frak{u}) = \Delta _{M,\mathcal{L}}([B]-f_{*}([A]))=\begin{cases}
\displaystyle 
c_{K}[\lambda _{K}]-(\underset{K'\subset L_{1}\setminus K}{\sum}{\rm lk}(K,K')c_{K'})[\mu _{K}] & (K\subset L_{1}) \\
\displaystyle 
-\underset{K'\subset L_{1}}{\sum}{\rm lk}(K,K')c_{K'}[\mu _{K}] &(K\subset \mathcal{L}\setminus L_{1}).\end{cases}
$$
%}
Thus, there exists $ K\subset L_{1}\subset \mathcal{L}$ such that $c_K [\lambda_K] \neq 0$.
On the other hand, by Lemma 1.10 for $\mathfrak{u}\in U^{f^{-1}(L_0)}=\underset{J \subset \mathfrak{L}\setminus f^{-1}(L_0)}{\prod}\mathbb{Z}[\mu _J], \mbox{ we obtain }
f_{*}(\mathfrak{u})\in \underset{K\subset \mathcal{L}}{\prod}\mathbb{Z}[\mu _{K}]$. This is a contradiction.

Thus, we have $[B]-f_{*}([A])= 0$, and hence $f_*(\frak{a}) = f_*(\alpha) \in f_*(P_{N,\frak{L}})$ as desired. %This completes the proof.
\end{proof}

\begin{rem} 
An integral homology 3-sphere ($\mathbb{Z}$HS$^3$) is an analogue of a number field with class number 1. A rational homology 3-sphere ($\mathbb{Q}$HS$^3$) is an analogue of a number field with any class number. We may extend our result to the cases over a general $\mathbb{Q}$HS$^3$ by introducing generalized notions of the linking numbers and Seifert surfaces. This will be discussed in our future study. 
\end{rem} 

\noindent 
\textbf{Acknowledgement}.  I would like to thank my supervisor Professor Masanori Morishita for proposing the problem to find a topological analogue for 3-manifolds of the Hasse norm principle, and for his advice and encouragement. Also, I would like to thank Jun Ueki for answering my question and for his advice.

\bibliographystyle{plainurl}%{amsplain}%{amsalpha}
\bibliography%{Tashiro.refs.bib}
{Tashiro_HNP_20240410.arXiv.bbl}

\begin{thebibliography}{1}

\bibitem{Hasse1931Beweis}
Helmut Hasse.
\newblock {B}eweis eines {S}atzes und {W}iderlegung einer {V}ermutung {\"u}ber
  das allgemeine {N}ormenrestsymbol ({P}roof of a theorem and disproof of a
  conjecture on the general norm residue symbol).
\newblock {\em Nachr. Ges. Wiss. G{\"o}ttingen, Math.-Phys. Kl.}, 1931:64--69,
  1931.

\bibitem{Mihara2019Canada}
Tomoki Mihara.
\newblock Cohomological approach to class field theory in arithmetic topology.
\newblock {\em Canad. J. Math.}, 71(4):891--935, 2019.
\newblock \href {http://dx.doi.org/10.4153/cjm-2018-020-0}
  {\path{doi:10.4153/cjm-2018-020-0}}.

\bibitem{Morishita2012}
Masanori Morishita.
\newblock {\em Knots and primes}.
\newblock Universitext. Springer, London, 2012.
\newblock An introduction to arithmetic topology.
\newblock \href {http://dx.doi.org/10.1007/978-1-4471-2158-9}
  {\path{doi:10.1007/978-1-4471-2158-9}}.

\bibitem{Neukirch}
J{\"u}rgen Neukirch.
\newblock {\em Algebraic number theory}, volume 322 of {\em Grundlehren der
  Mathematischen Wissenschaften [Fundamental Principles of Mathematical
  Sciences]}.
\newblock Springer-Verlag, Berlin, 1999.
\newblock Translated from the 1992 German original and with a note by Norbert
  Schappacher, With a foreword by G. Harder.
\newblock \href {http://dx.doi.org/10.1007/978-3-662-03983-0}
  {\path{doi:10.1007/978-3-662-03983-0}}.

\bibitem{Niibo1}
Hirofumi Niibo.
\newblock Id\`elic class field theory for 3-manifolds.
\newblock {\em Kyushu J. Math.}, 68(2):421--436, 2014.
\newblock \href {http://dx.doi.org/10.2206/kyushujm.68.421}
  {\path{doi:10.2206/kyushujm.68.421}}.

\bibitem{NiiboUeki}
Hirofumi Niibo and Jun Ueki.
\newblock Id\`elic class field theory for 3-manifolds and very admissible
  links.
\newblock {\em Trans. Amer. Math. Soc.}, 371(12):8467--8488, 2019.
\newblock \href {http://dx.doi.org/10.1090/tran/7480}
  {\path{doi:10.1090/tran/7480}}.

\bibitem{NiiboUeki2023RMS}
Hirofumi Niibo and Jun Ueki.
\newblock A {H}ilbert reciprocity law on 3-manifolds.
\newblock {\em Res. Math. Sci.}, 10(1):Paper No. 3, 8, 2023.
\newblock \href {http://dx.doi.org/10.1007/s40687-022-00364-w}
  {\path{doi:10.1007/s40687-022-00364-w}}.

\end{thebibliography}


\begin{thebibliography}{9}
\bibitem[Ha]{h}

H.Hasse, Beweis eines Satzes und Wiederlegung einer Vermutung \"{u}ber das allgemeine Normenrestsymbol, Nachrichten von der Gesellschaft der Wissenschaften zu G\"{o}ttingen, Mathematisch-Physikalische Klasse(1931), 64-69.

\bibitem[Mi]{mihara}

Tomoki Mihara, Cohomological approach to class field theory in arithmetic topology, Canad. J. Math. 71 (2019), no. 4, 891--935. %MR 3984024

\bibitem[Mo]{m}

M. Morishita, Knots and Primes -- An Introduction to Arithmetic Topology, Universitext, 2nd edition, Universitext, Springer, 2024.

\bibitem[Ne]{ne}

J. Neukirch, Algebraic Number Theory (Grundlehren Der Mathematiscen Wissenschaften, 322),Springer, Berlin,1999.

\bibitem[Ni]{n}

H. Niibo, Id\`elic class field theory for 3-manifolds, Kyushu J.Math, \textbf{68} (2014), no.2,421-436.

\bibitem[NU]{nu}

H. Niibo and J. Ueki, Id\`{e}lic class field theory for  3-manifolds and very admissible links, Transactions of the AMS, \textbf{371}, No.12, (2019), 8467-8488.

\bibitem[NU2]{nu2}
\underline{\ \ \ \ \ \ }, A Hilbert reciprocity law on 3-manifolds, Res. Math. Sci. 10 (2023), no. 1, Paper No. 3, 8. %MR 4519217

\end{thebibliography}

\begin{comment}
\begin{center}

\end{center} 
\end{comment} 

\

%\section*{Statement and Declarations} 
%Competing Interests: On behalf of all authors, the corresponding author states that there is no conflict of interest. 
%Data availability statements: This manuscript has no associated data. 

\leavevmode\\
Hirotaka Tashiro\\
Faculty of Mathematics, Kyushu University\\
744, Motooka, Nishi-ku, Fukuoka, 819-0395, JAPAN\\
e-mail: tashiro.hirotaka.035@s.kyushu-u.ac.jp

\end{document}